\documentclass{article}
\usepackage[english]{babel}
\usepackage[utf8]{inputenc}
\usepackage[T1]{fontenc}
\usepackage{geometry}
\usepackage{sectsty}
\usepackage{amsthm,amsfonts}
\usepackage{amsmath}
\usepackage{mathabx}
\sectionfont{\large}
\newtheorem{teo}{Theorem}
\newtheorem{lema}[teo]{Lemma}
\newtheorem{cor}[teo]{Corollary}
\newtheorem{prop}[teo]{Proposition}
\newtheorem{defi}[teo]{Definition}
\newtheoremstyle{mytheoremstyle} 
    {\topsep}                    
    {\topsep}                    
    {}                   
    {}                           
    {\scshape}                   
    {.}                          
    {.5em}                       
    {}  

\theoremstyle{mytheoremstyle} \newtheorem{nota}{Remark}
\numberwithin{equation}{section}
\newcommand{\real}{\mathbb{R}}
\newcommand{\nat}{\mathbb{N}}
\newcommand \ben {\begin{equation}}
\newcommand \een {\end{equation}}
\newcommand \be {\begin{equation*}}
\newcommand \ee {\end{equation*}}
\newcommand \bi {\begin{itemize}}
\newcommand \ei {\end{itemize}}

\date{}
\title{\textbf{Characterization of ground-states for a system of \textit{M} coupled semilinear Schrödinger equations and applications}}
\author{Simão Correia\\ \textit{CMAF-UL and FCUL, Av.\ Prof.\ Gama Pinto 2,}\\
\textit{1649-003 Lisboa, Portugal}\\
\textit{Email adress: sfcorreia@fc.ul.pt}}
\begin{document}
\maketitle

\begin{abstract}
We focus on the study of ground-states for the system of $M$ coupled semilinear Schrödinger equations with power-type nonlinearities and couplings. General results regarding existence and characterization are derived using a variational approach. We show the usefulness of such a characterization in several particular cases, including those for which uniqueness of ground-states is already known. Finally, we apply the results to find the optimal constant for the vector-valued Gagliardo-Nirenberg inequality and we study global existence, $L^2$-concentration phenomena and blowup profile for the evolution system in the $L^2$-critical power case.
\end{abstract}

\begin{section}{Introduction}
In this work, we consider the system of $M$ coupled semilinear Schrödinger equations
\ben\tag{M-NLS}
i(v_i)_t + \Delta v_i + \sum_{j=1}^M k_{ij}|v_j|^{p+1}|v_i|^{p-1}v_i=0,\quad i=1,...,M
\een
where $V=(v_1,...,v_M):\real^+\times\real^N\to\real^M$, $k_{ij}\in\real$, $k_{ij}=k_{ji}$, and $0<p<4/(N-2)^+$ (we use the convention $4/(N-2)^+=+\infty$, if $N=1,2$, and $4/(N-2)^+=4/(N-2)$, if $N\ge 3$). Given $1\le i\neq j\le M$, if $k_{ij}\ge 0$, one says that the coupling between the components $v_i$ and $v_j$ is attractive; if $k_{ij}< 0$, it is repulsive. The Cauchy problem for $V_0\in (H^1(\real^N))^M$ is locally well-posed and, letting $T_{max}(V_0)$ be the maximal time of existence of the solution with initial data $V_0$: if $T_{max}(V_0)<\infty$, then $\lim_{t\to T_{max}(V_0)}\|\nabla V(t)\|_2=+\infty$.

When we look for nontrivial periodic solutions of the form $V=e^{it}U$, with $U=(u_1,...,u_M)\in (H^1(\real^N))^M$ (called bound-states), we are led to the study of the system
\begin{equation}\label{BS}
 \Delta u_i - u_i + \sum_{j=1}^M k_{ij}|u_j|^{p+1}|u_i|^{p-1}u_i=0 \quad i=1,...,M.
\end{equation}

Especially relevant, for both physical and mathematical reasons, are the bound-states which have minimal action among all bound-states, the so-called ground-states. In the scalar case, one may prove that there is a unique ground-state (modulo translations and rotations). Examples of its relevance may be found, for example, in \cite{weinstein2}, \cite{merleraphael2}, \cite{cazenave}.

In the vector-valued case, very little is known. In fact, despite several results for the existence of bound-states, there are almost no results concerning ground-states and their characterization. To our knowledge, only the papers \cite{mazhao}, \cite{weiyao}, \cite{linwei} present advances in the characterization of ground-states, where the results obtained are quite specific. The approach for the first two is an analysis of the system of ODE's that one obtains after proving that all ground-states are radial functions. In the third paper, the approach is variational and offers only conditions for the existence (or nonexistence) of ground-states with all components different from zero. However, each of these results display several restrictions, both on the power $p$ and on the coefficients $k_{ij}$. 

Our approach is also variational, does not make restrictions on $p$ and is valid if the system \eqref{BS} has the following property: it is possible to group the components in such a way that two components attract each other if and only if they are in the same group. This property is verifiable in all the refered papers. Intuitively, the results tell us that the attractive components have the same profile and, if there are repulsive components, one of them has to be zero: otherwise, it would be possible to move them away from each other indefinetly and therefore lowering the action, which would contradict the minimality of the ground-state.

We call the reader's attention to theorem \ref{existencia}. A simple integration by parts shows that, if the matrix $\mathcal{K}=(k_{ij})$ is such that $X^T\mathcal{K}X\le 0$, for any $X\in \real^M$ with nonnegative components, there are no bound-states. Theorem \ref{existencia} claims that, if $\mathcal{K}$ does not satisfy this property, then there exist ground-states of (M-NLS). Therefore, this is the optimal result for the existence of ground-states of (M-NLS). The main difference regarding the known existence results is that, instead of using Schwarz symmetrization (for which one needs the positivity of the coupling coefficients), one uses the concentration-compactness principle by P.-L. Lions. Notice that this approach does not say wether there exist radial ground-states or not. However, in conjuction with the characterization theorems, we prove radiallity of ground-states in all the cases where it would be possible to use Schwarz symmetrization.

The structure of this work is as follows: in section 2, we define precisely the concepts of bound-state and ground-state and formulate the main results. The main lemma that allows the characterization in the case of attractive couplings can be set in a general framework, which we present in section 3. In section 4, we prove the main results. In section 5, we apply the results to some special cases, obtaining in particular the results of \cite{mazhao} and \cite{weiyao}. We also prove the uniqueness of ground-state in the case considered in \cite{maschulze}. Finally, in section 6, we use the characterization of ground-states to determine the optimal constant for the vector-valued Gagliardo-Nirenberg inequality and we apply the result to the study of global existence, concentration phenomena and blow-up profile for the (M-NLS) system, in the critical case $p=2/N$.

\end{section}
\begin{section}{Definitions and main results}

\begin{defi}\textit{(Bound-states and ground-states of (M-NLS))}
\begin{enumerate}

\item We define bound-state of (M-NLS) as any element $(u_1,...,u_M)\in (H^1(\real^N))^M\setminus \{0\}$ solution of \eqref{BS}
and define $A_{\footnotesize{\mbox{(M-NLS)}}}$ to be the set of all bound-states of (M-NLS).
\item A fully nontrivial bound-state is a bound-state such that $u_i\neq0, \ \forall i$. The set of such bound-states is called $A_{\footnotesize{\mbox{(M-NLS)}}}^+$.
\item Given $U=(u_1,...,u_M)\in (H^1(\real^N))^M$, set
\ben
I_M(U)=\sum_{i=1}^M \int |\nabla u_i|^2 + \int |u_i|^2, \ J_M(U)=\sum_{i,j=1}^M k_{ij}\int |u_i|^{p+1}|u_j|^{p+1}
\een
and define the action of $U$,
\ben
S_M(U)=\frac{1}{2}I_M(U) - \frac{1}{2p+2}J_M(U).
\een
\item The set of ground-states of (M-NLS) is defined as
\ben
G_{\footnotesize{\mbox{(M-NLS)}}}=\{U\in A_{\footnotesize{\mbox{(M-NLS)}}}: S_M(U)\le S_M(W),\ \forall W\in A_{\footnotesize{\mbox{(M-NLS)}}}\}\subset A_{\footnotesize{\mbox{(M-NLS)}}},
\een
and the set of fully nontrivial ground-states is
\ben
G_{\footnotesize{\mbox{(M-NLS)}}}^+=G_{\footnotesize{\mbox{(M-NLS)}}}\cap A_{\footnotesize{\mbox{(M-NLS)}}}^+.
\een
\end{enumerate}
\end{defi}

\begin{nota}
If $U\in A_{\footnotesize{\mbox{\textit{(M-NLS)}}}}$, $I_M(U)=J_M(U)$ (one multiplies the $i$-th equation by $u_i$ and integrates over $\real^N$). Therefore
\ben
S_M(U)=\left(\frac{1}{2}-\frac{1}{2p+2}\right), \quad I_M(U)=\left(\frac{1}{2}-\frac{1}{2p+2}\right)J_M(U).
\een
Hence a ground-state is a bound-state with $I_M$ (or $J_M$) minimal.
\end{nota}
\begin{nota}
Throughout this work, we shall assume that $k_{ij}$ are such that
\be\tag{P1}
\{U\in (H^1(\real^N))^M: J_M(U)>0\}\neq\emptyset.
\ee
This hypothesis is necessary for the existence of bound-states, since $J_M(U)=I_M(U)>0$, for any $U\in A_{\footnotesize{\mbox{\textit{(M-NLS)}}}}$.
\end{nota}
\begin{nota}
Since $M\ge 2$ will always be fixed, to simplify notations, we write
\ben
A:=A_{\footnotesize{\mbox{\textit{(M-NLS)}}}},\ G:=G_{\footnotesize{\textit{\mbox{(M-NLS)}}}},\ G^+:=G^+_{\footnotesize{\mbox{\textit{(M-NLS)}}}}
\een
and
\ben
I:=I_M,\ J:=J_M,\ S:=S_M.
\een
\end{nota}

The following two lemmas are  well-known results concerning ground-states for (1-NLS) (see \cite{cazenave}).

\begin{lema}\label{unicidadeQ}
There exists $Q\in H^1(\real^N)\setminus\{0\}$ radial, positive and strictly decreasing such that
\ben
G_{\footnotesize{\mbox{(1-NLS)}}}=\{e^{i\theta}Q(\cdot + y): \theta\in\real,\ y\in\real^N\}.
\een
\end{lema}

\begin{lema}\label{caracterizacaoQ}
$G_{(1-NLS)}$ is the set of solutions of the minimization problem
\ben
I_1(u)=\min_{J_1(w)=J_1(Q)} I_1(w),\quad J_1(u)=J_1(Q).
\een
\end{lema}

We now state the main results of this paper.
\begin{teo}\label{existencia}
Under assumption (P1), $G\neq \emptyset$.
\end{teo}

\begin{teo}\label{atractivoM}
Suppose (P1) and that $k_{ij}\ge 0,\ \forall i\neq j$. Then $U^0\in G^+$ if and only if there exist $\theta_i\in\real$, $i=1,...,M$, and $y\in\real^N$ such that
\ben\label{formula}
U^0=(a_ie^{i\theta_i}Q(\cdot+y))_{1\le i\le M}
\een
where
\ben
(a_1,...,a_M)\in S^+=\left\{B=(b_1,...,b_M)\in(\real^+)^M: \sum_{j=1}^M k_{ij}b_i^{p-1}b_j^{p+1}=1, i=1,...,M\right\}
\een
and
\ben\label{mina_0b_0}
\sum_{i=1}^M a_i^2I_1(Q) = \min\left\{\min_{B\in S^+}\left\{ \sum_{i=1}^M b_i^2I_1(Q) \right\}, \min_{U\in G\setminus G^+} I(U)\right\}.
\een
\end{teo}
\begin{nota}
Even though the result only characterize, \textit{a priori}, the elements of $G^+$, one may obtain the description of $G$. Simply notice that, if $U^0\in G\setminus G^+$, then
$U^0$ has $L$ nonzero components, with $1\le L<M$. If $(U^0)^+$ is the vector formed by such components, $(U^0)^+$ has to be a ground-state of a (L-NLS) system. By theorem \ref{atractivoM} applied with $M=L$, we find the explicit expression of $(U^0)^+$ and therefore of $U^0$.
\end{nota}

\begin{teo}\label{repulsivoM}
Suppose (P1) and that there exists a partition $\{Y_k\}_{1\le k\le K}$ of $\{1,...,M\}$ such that, given $1\le i\neq j\le M$,
\ben
k_{ij} \ge 0 \mbox{ if and only if } \exists k: i,j\in Y_k.
\een
Then, if $U^0=(u_1^0,...,u^0_M)\in G$, there exists $k\in\{1,...,K\}$ such that $u^0_i=0, \forall i\notin Y_{k}$.
\end{teo}
\begin{nota}
In the conditions of theorem \ref{repulsivoM}, we can also characterize the set $G$, since the vector of the nonzero components of a given ground-state of (M-NLS) is a ground-state for a (L-NLS) system, with $L<M$, where all the coupling coefficients are nonnegative. Therefore it is possible to apply theorem \ref{atractivoM} to (L-NLS), and thus obtaining the description of the initial ground-state.
\end{nota}

\begin{nota}
One may also consider solutions of (M-NLS) of the form $V(t)=(e^{i\omega_it}Q_i)_{1\le i\le M}$, $\omega_i>0$, and define bound-states and ground-states by making the appropriate changes. Our results of existence of ground-states can be easily extended to such a case, since one still has the homogeneity property for the functional $I$. The characterization results only extend to the simple case $\omega_i=\omega$, since lemma \ref{abstracto} requires that $I$ is the sum of several $I_1$'s (and not just a linear combination of them).
\end{nota}

\end{section}

\begin{section}{A general lemma}
Given a real vector space $X$, consider operators $I_1,J_1:X\to\real$ and $C:X\times X\to\real$ such that
\newline

(H1) $I_1$ is homogeneous of degree $\alpha>0$;
\newline

(H2) $J_1$ is homogeneous of degree $2\beta>0$ and $J_1(w)>0$ if $w\neq0$;
\newline

(H3) $C(\eta w,\xi w)=\eta^\beta \xi^\beta J_1(w)$ and $C(w,z)\le J_1(w)^{1/2}J_1(z)^{1/2}$, $\forall\ w,z\in X$ $\forall\ \eta,\xi>0$.
\newline

Given $c_{ij}\in \real,\ 1\le i,j\le M$, with $c_{ij}\ge 0$ if $i\neq j$, we define $$I(U):=\sum_{i=1}^M I_1(u_i) \mbox{ and } J(U) := \sum_{i,j=1}^M c_{ij}C(u_i,u_j).$$ 

\begin{lema}\label{abstracto}
Fix $\gamma>0$. Suppose that the family $\mathcal{M}\subset X$ of solutions of the minimization problem
\ben
I_1(u)=\min_{J_1(w)=\gamma} I_1(w),\quad J_1(u)=\gamma
\een
is nonempty and that $U=(u_1,...,u_M)\in (X\setminus\{0\})^M$ is a solution of the minimization problem
\ben
I(W)=\min_{J(V)\ge J(U)} I(V),\quad J(W)\ge J(U).
\een
Then there exist $d_i>0$ and $P_i\in\mathcal{M}$ such that $U=(d_iP_i)_{1\le i\le M}$.
\end{lema}
\begin{proof}
Let $R\in\mathcal{M}$.
First of all, we have
\ben
J_1\left(\left(\frac{J_1(R)}{J_1(u_i)}\right)^{\frac{1}{2\beta}}u_i\right)=J_1(R),\ 1\le i\le M
\een
Suppose, by absurd, and without loss of generality, that $d_1u_1\neq P, \forall d_1>0$ $\forall P\in \mathcal{M}$. By the minimality of $\mathcal{M}$,
\ben
I_1\left(\left(\frac{J_1(R)}{J_1(u_1)}\right)^{\frac{1}{2\beta}}u_1\right)>I_1(R)
\een
and
\ben
I_1\left(\left(\frac{J_1(R)}{J_1(u_i)}\right)^{\frac{1}{2\beta}}u_i\right)\ge I_1(R), \  2\le i\le M
\een
This implies that
\begin{align*}
I(U)&=I_1(u_1)+\sum_{i=2}^M I_1(u_i)> I_1\left(\left(\frac{J(u_1)}{J_1(R)}\right)^{\frac{1}{2\beta}}R\right) + \sum_{i=2}^M I_1\left(\left(\frac{J_1(u_i)}{J_1(R)}\right)^{\frac{1}{2\beta}}R\right)\\&=I\left(\left(\frac{J_1(u_1)}{J_1(R)}\right)^{\frac{1}{2\beta}}R,..., \left(\frac{J_1(u_M)}{J_1(R)}\right)^{\frac{1}{2\beta}}R\right).
\end{align*}
By the minimality of $U$,
\ben
J\left(\left(\frac{J_1(u_1)}{J_1(R)}\right)^{\frac{1}{2\beta}}R,..., \left(\frac{J_1(u_M)}{J_1(R)}\right)^{\frac{1}{2\beta}}R\right)<J(U).
\een
Using the definition of $J$,
\ben
\sum_{i,j=1, i\neq j}^M c_{ij}C\left(\left(\frac{J_1(u_i)}{J_1(R)}\right)^{\frac{1}{2\beta}}R, \left(\frac{J_1(u_j)}{J_1(R)}\right)^{\frac{1}{2\beta}}R\right)<\sum_{i,j=1, i\neq j}^M c_{ij}C(u_i,u_j)\le \sum_{i,j=1, i\neq j}^M c_{ij}J_1(u_i)^{\frac{1}{2}}J_1(u_j)^{\frac{1}{2}}.
\een
However, by the homogeneity of $C$,
\ben
\sum_{i,j=1, i\neq j}^M c_{ij}C\left(\left(\frac{J_1(u_i)}{J_1(R)}\right)^{\frac{1}{2\beta}}R, \left(\frac{J_1(u_j)}{J_1(R)}\right)^{\frac{1}{2\beta}}R\right)= \sum_{i,j=1, i\neq j}^M c_{ij}J_1(u_i)^{\frac{1}{2}}J_1(u_j)^{\frac{1}{2}},
\een
which is absurd.
\end{proof}
\end{section}
\begin{section}{Proof of the main results}
In this section, we fix $X=H^1(\real^N)$ and we adopt the definitions of section 2. Given $w,z\in H^1(\real^N)$, define
\ben
I_1(w)=\int |\nabla w|^2 + |w|^2,\ J_1(w)=\int |w|^{2p+2},\ C(w,z)=\int |w|^{p+1}|z|^{p+1}.
\een
It is easy to check that $I_1,J_1$,  $C$ satisfy (H1)-(H3).

\begin{nota}
Given $\lambda>0$, let
\ben
I^\lambda=\inf_{J(U)=\lambda} I(U) >0.
\een
By the homogenous property of $I$ and $J$, one easily checks that $I^\lambda=\lambda^{\frac{1}{p+1}}I^1$.
\end{nota}

Let
\ben
\lambda_G:=\left(\inf_{J(U)=1} I(U)\right)^{\frac{p+1}{p}}.
\een

\begin{lema}\label{igual=maior}
The minimization problems
\ben\label{minlambdaG}
I(U)=\min_{J(W)=\lambda_G} I(W),\quad J(U)=\lambda_G
\een
and
\ben\label{maior}
I(U)=\min_{J(W)\ge\lambda_G} I(W),\quad J(U)\ge\lambda_G
\een
are equivalent.
\end{lema}
\begin{proof}
Let $U^0$ be a solution of \eqref{maior}. If $J(U^0)>\lambda_G$, there would exist $c<1$ such that $J(cU^0)=\lambda_G$ and $I(cU^0)=c^2I(U^0)<I(U^0)$, contradicting the minimality of $U^0$. Hence $U^0$ is a solution of \eqref{minlambdaG}.

Now let $U^0$ be a solution of \eqref{minlambdaG}. If there existed $W$ with $J(W)\ge \lambda_G$ and $I(W)<I(U_0)$, then, for some $c\le 1$, $J(cW)= \lambda_G$ and, from the minimality of $U_0$, $I(U_0)\le I(cW) \le I(W) <I(U_0)$, which is absurd.
\end{proof}

\begin{lema}\label{caracterizacao}
Suppose that there exists a solution of the problem \eqref{minlambdaG}. Then $G$ is the set of solutions for \eqref{minlambdaG}.
\end{lema}
\begin{proof}

Let $U$ be a minimizer of \eqref{minlambdaG}. Then, for some $\mu\in\real$ and any $H=(h_1,...,h_M)\in (H^1(\real^N))^M$,
\ben
\langle -\Delta u_i + u_i, h_i\rangle_{H^{-1}\times H^1} = \mu(p+1) \langle \sum_{j=1}^M k_{ij}|u_j|^{p+1}|u_i|^{p-1}u_i, h_i\rangle_{H^{-1}\times H^1},\  1\le i\le M.
\een
Taking $H=U$, 
\ben
\lambda_G^{\frac{1}{p+1}}I^1 =I^{\lambda_G}=I(U) = \mu(p+1)  J(U)=\mu(p+1) \lambda_G
\een
The definition of $\lambda_G$ implies that $\mu(p+1) =1$ and so $U\in A$. Therefore
\ben
I(U)=\lambda_G \mbox{ and } S(U)=\left(\frac{1}{2}-\frac{1}{2p+2}\right)\lambda_G.
\een
Now we take $W\in A$. We want to see that $S(W)\ge S(U)$. Let $\gamma=J(W)$. Then
\ben
I(W)=\gamma \mbox{ and } S(W)=\left(\frac{1}{2}-\frac{1}{2p+2}\right)\gamma.
\een
Set $X=\left(\frac{\lambda_G}{\gamma}\right)^{\frac{1}{2p+2}}W$. Then $J(X)=\lambda_G$. Since $U$ is a minimizer of \eqref{minlambdaG},
\ben
\lambda_G^{\frac{1}{p+1}}I^1 = I(U)\le I(X) = \left(\frac{\lambda_G}{\gamma}\right)^{\frac{1}{p+1}}I(W)=\left(\frac{\lambda_G}{\gamma}\right)^{\frac{1}{p+1}}\gamma
\een
and so $\gamma\ge (I^1)^{\frac{p+1}{p}}=\lambda_G$. Hence
\ben
S(W)=\left(\frac{1}{2}-\frac{1}{2p+2}\right)\gamma\ge\left(\frac{1}{2}-\frac{1}{2p+2}\right)\lambda_G=S(U),
\een
which implies $U\in G$. If $W\in G$, one must have equality in the above inequality. Then
$J(W)=\lambda_G$ and, since $U,W\in A$, $I(W)=J(W)=J(U)=I(U)$. Therefore $W$ is a minimizer of $I^{\lambda_G}$.
\end{proof}

\noindent\textbf{\textit{Proof of theorem \ref{existencia}}}:

By lemma \ref{caracterizacao}, it suffices to prove that \eqref{minlambdaG} has a solution.

Let $\{U_n\}$ be a minimizing sequence of \eqref{minlambdaG}. Fix $\epsilon=\frac{1}{m}, m\in\nat$. In what follows, $\delta(\epsilon)$ shall be a function that goes to  $0$ when $\epsilon\to0$. Through the concentration-compactness principle of P.L.Lions (\cite{pllions1}, \cite{pllions2}), up to a subsequence, it is possible to associate to each $(U_n)_i$, $1\le i\le M$, a set of functions $\{(U_n)_i^l, (W_n)_i\}_{1\le l\le L}\subset H^1(\real^N)$ (a set of bubbles plus a remainder), such that
\begin{enumerate} 
\item Each $(U_n)_i^l$ has support in a ball of radius $R$ and the distance between the supports of $(U_n)_i^l$ and $(U_n)_i^j$, $j\neq l$, goes to $\infty$ as $n\to \infty$;
\item One has 
\ben
\left|\|(U_n)_i\|_{2p+2}^{2p+2} - \sum_{l=1}^L \|(U_n)_i^l\|_{2p+2}^{2p+2}\right|<\delta(\epsilon)
\een
and
\ben
\|\nabla (U_n)_i\|_2^2 \ge \sum_{l=1}^L \|\nabla (U_n)_i^l\|_2^2 -\delta(\epsilon), \ \|(U_n)_i\|_2^2 \ge \sum_{l=1}^L \|(U_n)_i^l\|_2^2 -\delta(\epsilon)
\een
\end{enumerate}
Essentially, one applies successively the concentration-compactness principle to each sequence $\{(U_n)_i\}$ to obtain the various bubbles. This process ends since the total $L^2$ norm is finite and because one always picks up the bubble with greater $L^2$ norm, which implies that, after $L_i$ steps, the remainder $W_n$ has $L^{2p+2}$ norm smaller than $\epsilon$. Setting $L=\max\{L_i\}$, we define, for each $i$, $(U_n)_i^l=0$ if $L_i<l\le L$.

One easily sees that, up to a subsequence, it is possible to group the bubbles into several clusters in such a way that: each cluster has one and only one bubble from each sequence $\{(U_n)_i\}$; if the supports of two bubbles have a nonempty intersection, then they must belong to the same cluster. Obviously, we shall end up with $L$ clusters. Define $U_n^l$ as the vector of bubbles from the cluster $l$. Then
\ben
\sum_{i=1}^M \|(U_n)_i\|_2^2 \ge \sum_{l=1}^L \sum_{i=1}^M \|(U_n^l)_i\|_2^2-\delta(\epsilon)
\een
and
\ben
\sum_{i=1}^M \|(\nabla U_n)_i\|_2^2 \ge \sum_{l=1}^L \sum_{i=1}^M \|\nabla (U_n)_i\|_2^2-\delta(\epsilon).
\een
Due to the way we grouped the bubbles, we have
\ben
\left| J(U_n) - \sum_{l=1}^L J(U_n^l)\right|\le \delta(\epsilon).
\een
Up to a subsequence, we can define $\lambda_l:=\lim J(U_n^l)$, $1\le l\le L$.
Using a diagonalization process, we obtain, for each $n$, a decomposition of $\{U_n\}$ in $L_n$ bubbles (where $L_n\to \overline{L}\in \nat\cup \{\infty\}$) such that 
\ben
\sum_{i=1}^M \|(U_n)_i\|_2^2 \ge \sum_{l=1}^L \sum_{i=1}^M \|(U_n^l)_i\|_2^2-\delta\left(\frac{1}{n}\right),\quad 
\sum_{i=1}^M \|(\nabla U_n)_i\|_2^2 \ge \sum_{l=1}^L \sum_{i=1}^M \|\nabla (U_n)_i\|_2^2-\delta\left(\frac{1}{n}\right),
\een
\ben
\left| J(U_n) - \sum_{l=1}^{L_n} J(U_n^l)\right|\le \delta\left(\frac{1}{n}\right)
\een
and
\ben\label{somalambdas}
\lambda_G=\sum_{l=1}^{\overline{L}} \lambda_l
\een
\textit{Case 1:} If $\lambda_l\ge0$, for any $l$, one has
\ben
J\left(\left(\frac{\lambda_l}{J(U_n^l)}\right)^{\frac{1}{2p+2}}U_n^l\right)=\lambda_l
\een
and so
\ben
I^{\lambda_G} = \lim I(U_n) \ge \limsup \sum_{l= 1}^{L_n} \frac{J(U_n^l)}{\lambda_l}I\left(\left(\frac{\lambda_l}{J(U_n^l)}\right)^{\frac{1}{2p+2}}U_n^l\right) \ge \limsup \sum_{l= 1}^{L_n} I^{\lambda_l}= \sum_{l=1}^{\overline{L}} I^{\lambda_l}.
\een
However, the function
\ben
\lambda\mapsto I^{\lambda}=\lambda^{\frac{1}{p+1}}I^1
\een
is strictly concave in $\real^+$, which implies that there exists $l_0$ such that $\lambda_l=0$, for $l\neq l_0$. By \eqref{somalambdas}, $\lambda_{l_0}=\lambda_G$. Therefore, defining
\ben
W_n=\left(\frac{\lambda_G}{J(U_n^{l_0})}\right)^{\frac{1}{2p+2}}U_n^{l_0},
\een
one has
\ben
\liminf I(U_n)-I(W_n)\ge 0,\quad J(W_n)=\lambda_G
\een
and so $\{W_n\}$ is a minimizing sequence for \eqref{minlambdaG}, for which the compactness alternative from the concentration-compactness principle is verified (recall that $W_n$ is, up to a multiplicative factor, the vector of a group of bubbles of $U_n$). Since $\{W_n\}$ is bounded in $(H^1(\real^N))^M$, there exists $W\in (H^1(\real^N))^M$ such that $W_n\rightharpoonup W$ and, from the compactness alternative, it follows that $W_n\to W$ in $(L^2(\real^N)\cap L^{2p+2}(\real^N))^M$. In particular,
\ben
I(W)\le \lim I(W_n) = I^{\lambda_G}, \quad J(W)=\lim J(W_n)=\lambda_G.
\een
Therefore $W$ is a minimizer of \eqref{minlambdaG}.

\noindent\textit{Case 2}: Now suppose that
\ben
L^-=\{l: \lambda^l<0\}\neq\emptyset.
\een
Define $L^+$ to be the complementary set of $L^-$ and
\ben
\eta_l:=\frac{\sum_{j=1}^{\overline{L}} \lambda_j}{\sum_{l\in L^+} \lambda_j}\lambda_l.
\een
Notice that \eqref{somalambdas} implies $L^+\neq\emptyset$. Furthermore,
\ben
\lambda_G=\sum_{l\in L^+} \eta_l.
\een
Since
\ben
J\left(\left(\frac{\eta_l}{J(U_n^l)}\right)^{\frac{1}{2p+2}}U_n^l\right)=\eta_l, \  l\in L^+,
\een
one has
\ben
I^{\lambda_G} = \lim I(U_n) \ge \limsup \sum_{l\in L^+} \frac{J(U_n^l)}{\eta_l}I\left(\left(\frac{\eta_l}{J(U_n^l)}\right)^{\frac{1}{2p+2}}U_n^l\right) \ge \limsup \sum_{l= 1}^{L_n} I^{\eta_l} = \sum_{l=1}^{\overline{L}} I^{\eta_l}.
\een
We now conclude in the same way as the previous case. $\qedsymbol$
\vskip10pt

For the case where all components attract each other, one may improve the above result using Schwarz symmetrization. This fact is not new (see \cite{linwei}), however we display the following result for the sake of completeness.

\begin{prop}
If $k_{ij}\ge 0, \forall 1\le i\neq j\le M$, then \eqref{minlambdaG} has a positive, radial, decreasing solution.
\end{prop}
\begin{proof}
Let $\{U_n\}$ be a minimizing sequence of \eqref{minlambdaG}. Defining $|W|:=(|w_1|,...,|w_M|)$, clearly $\{|U_n|\}$ is also a minimizing sequence. Let $|W|^*=(|w_1|^*,...,|w_M|^*)$ be the vector of the Schwarz symmetrizations of the components of  $|W|$. The properties of the symmetrization imply that $\{|U_n|^*\}$ satisfies
\ben
J(|U_n|^*)\ge\lambda_G, \quad I^{\lambda_G}\le\liminf I(|U_n|^*)\le \lim I(U_n)=I^{\lambda_G}.
\een

Using a compactness result for Schwarz symmetrizations, up to a subsequence, $|U_n|^*\rightharpoonup U$ in $(H^1(\real^N))^M$ and $|U_n|^*\to U$ in $(L^2(\real^N)\cap L^{2p+2}(\real^N))^M$. Hence
\ben
J(U)=\lim J(|U_n|^*)\ge\lambda_G,\ \quad I^{\lambda_G}\le I(U)\le\liminf I(|U_n|^*) = I^{\lambda_G}.
\een
Therefore $U$ is a solution of \eqref{maior} and, by lemma \ref{igual=maior}, it is a solution of \eqref{minlambdaG}.
\end{proof}

\noindent\textit{\textbf{Proof of theorem \ref{atractivoM}:}}
We divide the proof in three steps:

\textit{Step 1:} $U^0\in G^+$ satisfies \eqref{formula}, with $A^0=(a_1,...,a_M)\in S^+$.

Let $U_0\in G^+$. 
By lemmata \ref{unicidadeQ}, \ref{caracterizacaoQ}, \ref{igual=maior} and \ref{caracterizacao},  we may apply lemma \ref{abstracto} to $I_1,J_1$ and $C$ and therefore we conclude that there exist, for each $1\le i\le M$, $a_i>0$, $\theta_i\in\real$ and $y_i\in \real^N$ such that
\ben
U^0=(a_ie^{i\theta_i}Q(\cdot + y_i))_{1\le i\le M}.
\een

If there exist $i_0, j_0$ such that $y_{i_0}\neq y_{j_0}$, one easily sees that there exists $D\subset \real^N$ of positive measure such that, for all $x\in D$, $Q(x+y_{i_0})\neq Q(x+y_{j_0})$ and so, using Young's inequality,
\ben
Q(x+y_{i_0})^{p+1}Q(x+y_{j_0})^{p+1}< \frac{1}{2}Q(x+y_{i_0})^{2p+2} + \frac{1}{2}Q(x+y_{j_0})^{2p+2},\quad x\in D.
\een
On the other hand, we have in general
\ben
Q(x+y_{i})^{p+1}Q(x+y_{j})^{p+1}\le \frac{1}{2}Q(x+y_{i})^{2p+2} + \frac{1}{2}Q(x+y_{j})^{2p+2},\quad x\in \real^N,\  1\le i,j\le M.
\een
Consequently,
\begin{align*}
\int (a_{i}Q(\cdot+y_{i}))^{p+1}(a_jQ(\cdot+y_j))^{p+1} \le a_i^{p+1}a_j^{p+1}\left(\frac{1}{2}\int Q(\cdot+y_i)^{2p+2} + \frac{1}{2}\int Q(\cdot+y_j)^{2p+2}\right)\\ = a_i^{p+1}a_j^{p+1}\int Q^{2p+2} = \int (a_iQ)^{p+1}(a_jQ)^{p+1},
\end{align*}
with strict inequality if $i=i_0$ and $j=j_0$.
Therefore, $\lambda_G=J(U^0)<J((a_i Q)_{1\le i\le M})=:\lambda$. Hence
\ben
J\left(\left(\frac{\lambda_G}{\lambda}\right)^{\frac{1}{2p+2}}(a_i Q)_{1\le i\le M}\right)=\lambda_G
\een
and
\ben
 I\left(\left(\frac{\lambda_G}{\lambda}\right)^{\frac{1}{2p+2}} (a_i Q)_{1\le i\le M}\right)<I\left((a_i Q)_{1\le i\le M}\right)=
I(U_0),
\een
which contradicts the minimality of $U_0$. Therefore $y_i=y_j$, for any $1\le i,j\le M$ and so $U_0$ is of the form \eqref{formula}.

Replacing the formula of $U_0$ into the system \eqref{BS}, we derive
\ben
\sum_{j=1}^M k_{ij} a_i^{p-1}a_j^{p+1}=1\ \forall 1\le i\le M.
\een
Hence $A_0\in S^+$.

\textit{Step 2:} If $U^0$ is of the form \eqref{formula}, with $A_0\in S^+$, $U^0\in A$.

Simply notice that $U^0$ satisfies the system \eqref{BS}, using the conditions of $S^+$.

\textit{Step 3:} Conclusion.

Let $U^0\in G^+$. If $A_0$ does not satisfy \eqref{mina_0b_0}, then either
\ben
\min_{U\in G\setminus G^+} I(U)<\sum_{i=1}^M a_i^2I_1(Q)=I(U^0)
\een
or there exists $B\in S^+$ such that
\ben
\sum_{i=1}^M b_i^2I_1(Q) < \sum_{i=1}^M a_i^2I_1(Q).
\een
In the first case, there would exist $U\in G\setminus G^+$ with $I(U)<I(U^0)$, which contradicts $U^0\in G$. In the second case, given $\theta_i\in\real$, $1\le i\le M$, and $y\in\real^N$,
\ben
W^0:=(b_ie^{i\theta_i}Q(\cdot + y))_{1\le i\le M}
\een
is in $A$. Moreover,

\begin{align*}
S(W^0)&=\left(\frac{1}{2}-\frac{1}{2p+2}\right)I(W^0)=\left(\frac{1}{2}-\frac{1}{2p+2}\right)\sum_{i=1}^M b_i^2I_1(Q) \\&< \left(\frac{1}{2}-\frac{1}{2p+2}\right)\sum_{i=1}^M a_i^2I_1(Q) = S(U_0),
\end{align*}
which contradicts  $U_0\in G$. We conclude that $A_0$ satisfies \eqref{mina_0b_0}. 
It remains to prove that $W^0\in G$. In fact,
\begin{align*}
S(W^0)&=\left(\frac{1}{2}-\frac{1}{2p+2}\right)I(W^0)=\left(\frac{1}{2}-\frac{1}{2p+2}\right)\sum_{i=1}^Mb_i^2I_1(Q) \\&= \left(\frac{1}{2}-\frac{1}{2p+2}\right)\sum_{i=1}^M a_i^2I_1(Q) = S(U^0).
\end{align*}
Therefore $W_0\in G$, which ends the proof. $\qedsymbol$
\vskip10pt

\noindent\textbf{\textit{Proof of theorem \ref{repulsivoM}:}}
The partition $\{Y_k\}_{1\le k\le K}$ defines an equivalence relation in the set $\{1,...,M\}$:
\ben
i\sim j \mbox{ if and only if } \exists k\ i,j\in Y_k.
\een

We claim that \eqref{minlambdaG} is equivalent to 
\ben\label{minlambdaGC=0}
\sum_{i=1}^M I(u_i)=\min_B \sum_{i=1}^M I(w_i),\quad (u_1,...,u_M)\in B
\een
where
\ben
B=\left\{(w_1,...,w_M)\in (H^1(\real^N))^M: \sum_{k=1}^K \sum_{i,j \in Y_k} k_{ij}C(w_i,w_j)=\lambda_G,\ C(w_i,w_j)=0 \mbox{ if } i\not\sim j\right\}.
\een
To see this, suppose that $U^0$ is a solution of \eqref{minlambdaG}. If $C(u_i,u_j)=0, \forall  i\not\sim j$, then $U^0$ is a solution of \eqref{minlambdaGC=0}. By absurd, suppose that there exist $i_0\not\sim j_0$ such that $C(u_{i_0},u_{j_0})\neq 0$. Let $U^R$ be defined by
\ben
(U^R)_i=(U^0)_i, \mbox{ if } i\not\sim j_0,\quad (U^R)_i=(U^0)_i(\cdot+Re_1)\mbox{ if } i\sim j_0.
\een
Then, for large $R$, $C((U^R)_i,(U^R)_j)\le C((U^0)_i,(U^0)_j)$ if $i\not\sim j$ (with strict inequality if $i=i_0$, $j=j_0$) and $C((U^R)_i,(U^R)_j)= C((U^0)_i,(U^0)_j)$ if $i\sim j$. Hence, $J(U^R)>J(U^0)$. Since
\ben
J\left(\left(\frac{J(U^0)}{J(U^R)}\right)^{\frac{1}{2p+2}}U^R\right) = J(U^R),
\een
we have, by the minimality of $U^0$,
\ben
I(U^0)\le I\left(\left(\frac{J(U^0)}{J(U^R)}\right)^{\frac{1}{2p+2}}U^0\right) = \left(\frac{J(U^0)}{J(U^R)}\right)^{\frac{1}{p+1}}I(U^0)<I(U^0),
\een
which is absurd.
On the other hand, if $U^0$ is a solution of \eqref{minlambdaGC=0}, suppose thet there exists  $W$ such that $J(W)=\lambda_G$ and $I(W)<I(U^0)$. If $C(w_i,w_j)=0, \forall  i\not\sim j$, we obtain, through the minimality of $U^0$, $I(U^0)\le I(W)$, which is absurd. If there exist $i_0\not\sim j_0$ such that $C(w_{i_0},w_{j_0})>0$, let $\xi:\real^N\to[0,1]$ be a smooth cutoff function with support on the unit ball and $\xi_R(x)=\xi(x/R)$. Define $W^R$ by
\ben
 (W^R)_i=\xi_Rw_i(\cdot + 2kRe_1),\mbox{ if } i\in Y_k
\een
and, for each $n\in\nat$, let $R_n$ be such that
\ben
|I_1((W^{R_n})_i) - I_1(w_i))|<\frac{1}{n},\ |J_1((W^{R_n})_i) - J_1(w_i))|<\frac{1}{n}.
\een
It is clear that
\ben
C((W^{R_n})_i,(W^{R_n})_j)=0, \ i\not\sim j;\ C((W^{R_n})_i,(W^{R_n})_j)=C(W_i,W_j), \ i\sim j
\een
and so
\ben
\limsup J(W^{R_n})\ge \lambda_G
\een
Therefore there exist $\lambda_n$, with $\liminf \lambda_n\le 1$, such that
\ben
 J(\lambda_nW^{R_n})= \lambda_G
\een
and, by the minimality of $U^0$,
$$I(U^0)\le \lim I(\lambda_nW^{R_n}) =I(W)<I(U^0),$$
which is absurd. Hence $U^0$ is a solution of \eqref{minlambdaG}. Thus the minimization problems \eqref{minlambdaG} and \eqref{minlambdaGC=0} are equivalent.
\vskip10pt

Let
\ben
K_G=\{k\in \{1,...,K\}: \exists U\in (H^1(\real^N))^M: \sum_{i,j\in Y_k} k_{ij}C(u_i,u_j)> 0\}.
\een

For $Z=(z_1,...,z_M)\in B$, define
\ben
K^+=\{ k\in \{1,...,K\}: \sum_{i,j\in Y_k} k_{ij}C(z_i,z_j)> 0\}\subset K_G,
\een
and $\overline{Z}$ as $\overline{Z}_i=z_i,\ i\in Y_k, k\in K^+$ and $\overline{Z}_i=0,\ i\in Y_k, k\notin K^+$.

Then
\ben
W:=\left(\frac{J(Z)}{J(\overline{Z})}\right)^{\frac{1}{2p+2}}\overline{Z}\in B\mbox{ and } I\left(\left(\frac{J(Z)}{J(\overline{Z})}\right)^{\frac{1}{2p+2}}\overline{Z}\right)\le I(\overline{Z})\le I(Z).
\een

with strict inequality if $Z\neq \overline{Z}$.


For each $1\le k\le K_G$, let $Q^k \in H^1(\real^N)^{|Y_k|}$ be a ground-state of the system formed by the equations of the  $i$-th components, with $i\in Y_k$. Fix $k\in K^+$. Then, defining
\ben
c_k=\left(\frac{\sum_{i,j\in Y_k} k_{ij}C(w_i,w_j)}{\sum_{i,j\in Y_k} k_{ij}C(Q_i^k,Q_j^k)}\right)^{\frac{1}{2p+2}}
\een
we have
\ben
\sum_{i,j\in Y_k} k_{ij}C\left(\frac{w_i}{c_k},\frac{w_j}{c_k}\right) = \sum_{i,j\in Y_k} k_{ij}C\left(Q^k_i,Q^k_j\right) .
\een
Since $Q^k$ is a solution of \eqref{minlambdaG}, with $M=|Y_k|$ and $i,j\in Y_k$, we obtain
\ben\label{Meq1}
\sum_{i\in Y_k} I_1(Q_i^k) \le \sum_{i\in Y_k} I_1\left(\frac{w_i}{c_k}\right)=\frac{1}{c_k^2}\sum_{i\in Y_k} I_1(w_i)
\een
and so
\begin{align*}
 \sum_{k\in K^+}^K c_k^2 \sum_{i\in Y_k} I_1(Q^k_i)\le \sum_{i=1}^M I_1(w_i).
\end{align*}
Let $k_0$ be such that
\ben
\frac{\left(\sum_{i\in k_0}I_1(Q_i^{k_0})\right)^{p+1}}{\sum_{i,j\in Y_{k_0}}k_{ij}C(Q_i^{k_0},Q_j^{k_0})}\le \frac{\left(\sum_{i\in k}I_1(Q_i^{k})\right)^{p+1}}{\sum_{i,j\in Y_{k}}k_{ij}C(Q_i^{k},Q_j^{k})}, \ \forall k\in K_G.
\een
Let $\mathcal{Q}\in (H^1(\real^N))^M$ be defined by $\mathcal{Q}_i=0$ if $i\notin Y_{k_0}$ and otherwise
\ben
\mathcal{Q}_i=\left(\frac{\lambda_G}{\sum_{i,j\in Y_{k_0}}k_{ij}C(Q_i^{k_0},Q_j^{k_0})}\right)^{\frac{1}{2p+2}}Q_i^{k_0}=\left(\frac{\sum_{k\in K^+}\sum_{i,j\in Y_k} k_{ij}C(w_i,w_j)}{\sum_{i,j\in Y_{k_0}}k_{ij}C(Q_i^{k_0},Q_j^{k_0})}\right)^{\frac{1}{2p+2}}Q_i^{k_0} =: d_{k_0}Q_i^{k_0}.
\een
It is easy to see that $J(\mathcal{Q})=\lambda_G$ and, by the definition of $k_0$,
\begin{align*}
I(\mathcal{Q})&=d_{k_0}^2 \sum_{i\in Y_{k_0}} I_1(Q_i^{k_0}) = \left(\sum_{k\in K^+} \sum_{i,j\in Y_k} k_{ij}C(w_i,w_j) \frac{\left(\sum_{i\in Y_{k_0}}I_1(Q_i^{k_0})\right)^{p+1}}{\sum_{i,j\in Y_{k_0}}k_{ij}C(Q_i^{k_0},Q_j^{k_0})}\right)^{\frac{1}{p+1}} \\
&\le \left(\sum_{k\in K^+}\frac{\sum_{i,j\in Y_k} k_{ij}C(w_i,w_j)}{\sum_{i,j\in Y_{k}}k_{ij}C(Q_i^{k},Q_j^{k})}\left(\sum_{i\in Y_k}I_1(Q_i^{k})\right)^{p+1}\right) ^{\frac{1}{p+1}}=\left(\sum_{k\in K^+} \left(c_k^2\sum_{i\in Y_k}I_1(Q_i^{k})\right)^{p+1}\right)^{\frac{1}{p+1}}\\
&\le \sum_{k\in K^+} c_k^2 \sum_{i\in Y_k} I_1(Q^k_i)\le \sum_{i=1}^M I_1(w_i) = I(W) \le I(Z).
\end{align*}
Therefore $\mathcal{Q}$ is a solution of \eqref{minlambdaGC=0} and, by lemma \ref{caracterizacao}, $\mathcal{Q}\in G$.
Finally, if $Z\in G$, then $I(Z)=I(\mathcal{Q})$, which implies that all of the above inequalities must be in fact equalities. From the above computation, we obtain, for some $1\le k_Z\le K$, $c_k=0$, $\forall k\neq k_Z$ and $\overline{Z}=Z$. Hence $K^+=\{k_Z\}$ and the proof is concluded. $\qedsymbol$

\end{section}

\begin{section}{Some special cases}

In this section, we apply the results to some special cases, obtaining in particular the results of \cite{mazhao} and \cite{weiyao}. We shall always suppose $k_{ij}\ge 0$, $i\neq j$. 

We start with $M=2$. Given $(u_0,v_0)\in G^+$, we note by $a_0, b_0$ the constants of the characterization from theorem \ref{atractivoM}. 

\begin{cor}
Suppose that $k_{11}= k_{22}\le 0$ and $k_{12}>-k_{11}$. Let $(u_0,v_0)\in G^+$. Then $a_0=b_0=(k_{11}+k_{12})^{-\frac{1}{2p}}$.
\end{cor}
\begin{proof}
By theorem \ref{atractivoM}, we know that
\ben
\left\{\begin{array}{l}
k_{11}a_0^{2p} + k_{12}a_0^{p-1}b_0^{p+1}=1\\
k_{22}b_0^{2p} + k_{12}b_0^{p-1}a_0^{p+1}=1
\end{array}\right..
\een
Suppose that $a_0\neq b_0$. By the symmetry of the system, it's enough to prove that $a_0\ge b_0$.

Multiplying the first equation by $a_0^2$, the second by $b_0^2$ and subtracting,
\ben
k_{11}a_0^{2p+2} - k_{22}b_0^{2p+2} = a_0^2 - b_0^2.
\een

If $a_0<b_0$, the left-hand side is nonnegative and the right one is negative, which is absurd. Therefore $a_0=b_0$. The value of $a_0$ can now be directly calculated from the system.
\end{proof}
\begin{cor}
Suppose that $p=1$ and $k_{ij}>0$, $i,j=1,2$. Then
\begin{enumerate}
\item If $k_{11}\neq k_{22}$ and $k_{11}\le k_{12}\le k_{22}$, $G^+=\emptyset$;
\item If $k_{12}\notin [\min\{k_{11},k_{22}\},\max\{k_{11},k_{22}\}]$ and $(u_0,v_0)\in G^+$, then
\ben\label{raizes}
a_0=\sqrt{\frac{k_{22}-k_{12}}{k_{11}k_{22}-k_{12}^2}},\quad b_0=\sqrt{\frac{k_{11}-k_{12}}{k_{11}k_{22}-k_{12}^2}}.
\een
Consequently, $G^+=\emptyset$ if $k_{12}<\min\{k_{11},k_{22}\}$ and $G^+=G$ if $k_{12}>\max\{k_{11},k_{22}\}$.
\item If $k_{11}=k_{12}=k_{22}$, $(u_0,v_0)\in G^+$ if and only if
\ben
(a_0,b_0)=\left(\frac{1}{\sqrt{k_{11}}}\cos \alpha, \frac{1}{\sqrt{k_{11}}}\sin \alpha\right), \ \alpha\in ]0,\pi/2[.
\een
\end{enumerate}
\end{cor}
\begin{proof}
By theorem \ref{atractivoM}, we know that
\ben
\left\{\begin{array}{l}
k_{11}a_0^{2} + k_{12}b_0^{2}=1\\
k_{22}b_0^{2} + k_{12}a_0^{2}=1
\end{array}\right.
\een
Therefore
\ben
\left\{\begin{array}{l}
(k_{11}k_{22} - k_{12}^2)a_0^2 = k_{22} - k_{12}\\
(k_{11}k_{22} - k_{12}^2)b_0^2 = k_{11} - k_{12}
\end{array}\right.
\een
\begin{enumerate}
\item If $k_{11}\neq k_{22}$ and $k_{11}\le k_{12}\le k_{22}$, suppose, without loss of generality, that $k_{11}<k_{12}$. Then
\ben
\frac{a_0^2}{b_0^2} = \frac{k_{22}-k_{12}}{k_{11}-k_{12}} \le 0
\een
which is absurd.

\item If $k_{12}\notin [\min\{k_{11},k_{22}\},\max\{k_{11},k_{22}\}]$, one can explicitly determine the values of $a_0$ and $b_0$, thus obtaining the formulas \ref{raizes}. Suppose, w.l.o.g., that $k_{11}\le k_{22}$. If $k_{12}<k_{22}$, then one easily checks that
$$
I(a_0Q,b_0Q)=\frac{k_{22}-k_{12}}{k_{11}k_{22}-k_{12}^2} + \frac{k_{11}-k_{12}}{k_{11}k_{22}-k_{12}^2}>\frac{1}{k_{22}}=I\left(\frac{1}{\sqrt{k_{22}}}Q\right).
$$
Therefore $G^+=\emptyset$. If $k_{12}>k_{22}$, the above inequality is reversed and one obtains $G=G^+$.
\item If $k_{11}=k_{12}=k_{22}$, then $a_0^2 + b_0^2 = 1/k_{11}$ and so there exists $\alpha\in ]0,\pi/2[$ such that
\ben
(a_0,b_0)=\left(\frac{1}{\sqrt{k_{11}}}\cos \alpha, \frac{1}{\sqrt{k_{11}}}\sin \alpha\right).
\een
On the other hand, any pair of this form is in $S^+$ and has minimal norm. The conclusion follows from theorem \ref{atractivoM}.
\end{enumerate}
\end{proof}

\begin{cor}
Suppose that $k_{11}=k_{22}>0$, $k_{12}>0$. and $(p-2)(pk_{11}-k_{12})>0$. If $(u_0,v_0)\in G^+$, then $a_0=b_0=(k_{11} + k_{12})^{-\frac{1}{2p}}$.
\end{cor}
\begin{proof}
Again by theorem \ref{atractivoM},
\ben
\left\{\begin{array}{l}
k_{11}a_0^{2p} + k_{12}a_0^{p-1}b_0^{p+1}=1\\
k_{22}b_0^{2p} + k_{12}b_0^{p-1}a_0^{p+1}=1
\end{array}\right.
\een
Taking the difference between the two equations and dividing by $b_0^{2p}$,
\ben
k_{11}\left(\frac{a_0}{b_0}\right)^{2p}  - k_{11} + k_{12}\left(\left(\frac{a_0}{b_0}\right)^{p-1} - \left(\frac{a_0}{b_0}\right)^{p+1}\right) = 0.
\een
Consider the function $f(x)=k_{11}x^{2p} - k_{11} + k_{12}(x^{p-1} - x^{p+1}),\ x>0$. It is clear that $f(1)=0$ and $f(0)<0$. We want to see that $f$ does not have zeroes on both sides of $1$.
One has
\begin{align*}
f'(x)&= 2pk_{11}x^{2p-1} + k_{12}((p-1)x^{p-2} - (p+1)x^p)=x^{p-2}\left(2pk_{11}x^{p+1} + k_{12}((p-1) - (p+1)x^2)\right)\\&=: x^{p-2}g(x)
\end{align*}
and
\ben
g'(x)=2p(p+1)k_{11}x^p - 2(p+1)k_{12}x.
\een
Clearly
\ben
g'(x)=0 \Leftrightarrow x=\left(\frac{k_{12}}{pk_{11}}\right)^{\frac{1}{p-1}}.
\een

Since $g'$ has a unique zero, $f$ has at most three (counting multiplicities), one of which $x=1$. If $p>2$, since $f(x)\to\infty$ when $x\to\infty$ and $f'(1)=g(1)=2(pk_{11}-k_{12})>0$, all the zeroes of $f$ have to be on the same side with respect to $x=1$, as we wanted. If $p<2$, since $f(x)\to -\infty$ when $x\to\infty$ and $f'(1)=g(1)=2(pk_{11}-k_{12})<0$, we obtain the same conclusion.

Suppose, without loss of generality, that $f$ has no zeroes on $]0,1[$. It follows that $f(x)=0$ implies $x\le1$ and so $a_0\le b_0$. By the symmetry of the system, $a_0\ge b_0$. Hence $a_0=b_0$. The value of $a_0$ can now be determined from the system.
\end{proof}

\begin{nota}
In the case $p<2$ and $pk_{11}-k_{12}>0$, one may easily check that the function $f$ in the above proof has three distinct zeroes $x_0,1$ and $x_0^{-1}$.
\end{nota}

To conclude this section, we prove the following result:
\begin{prop}\label{propnguyen}
Fix $M\ge 2$, $p\ge 2$ and suppose that, for each $1\le i\le M$, $k_{ii}>0$ and $k_{ij}\ge 0, j\neq i$. If $\beta=\max_{i\neq j} k_{ij}$ is sufficiently small, then, letting $\mathcal{I}$ be the set of $i_0$'s such that $k_{i_0i_0}^{-\frac{1}{p+1}}= \min_i k_{ii}^{-\frac{1}{p+1}}$ and, for any $i_0\in \mathcal{I}$, $\mathcal{Q}_{i_0}\in (H^1(\real^N))^M$ defined by $(\mathcal{Q})_i=0$ if $i\neq i_0$ and $(\mathcal{Q})_{i_0}=k_{i_0i_0}^{-\frac{1}{p+1}}Q$ (recall lemma \ref{unicidadeQ}), one has
\ben\label{caracG}
G=\{e^{i\theta}\mathcal{Q}_{i_0}(\cdot+y),\ i_0\in \mathcal{I}, \theta\in\real, y\in \real^N\}.
\een
\end{prop}
\begin{proof}
Set $A^1=(k_{ii}^{-\frac{1}{2p+2}})_{1\le i\le M}$ and
$\mathcal{S}_0$ the vector space of symmetrical matrices $M\times M$ with zero diagonal, equipped with the $l^\infty$ norm.
Consider $F:\mathcal{S}_0\times \real^M\to \real^M$,
\ben
F_i(D, A)= k_{ii}a_i^{2p+2} + \sum_{j=1,\ i\neq j}^M d_{ij}a_i^{p-1}a_j^{p+1},\ D=(d_{ij}), \ A=(a_i).
\een
Then $F(0,A^1)=0$, $F$ is $C^1$ and it is easy to see that the jacobian of $F$ with respect to $A$ in $A^1$ is nonzero. By the implicit function theorem, if $\|D\|_{\mathcal{S}_0}<\delta$, there exists a unique solution of $F(D,A)=1$, called $A(D)$, and there exists $\epsilon>0$ small enough such that $\|A(D)-A^1\|_{\real^M}<\epsilon$. Consequently
\ben
\sum_{i=1}^M (A(D))_i^2 \ge  \sum_{i=1}^M (A^1)_i^2 - \epsilon = \sum_{i=1}^M k_{ii}^{-\frac{1}{p+1}} -\epsilon > \min_i\{k_{ii}^{-\frac{1}{p+1}}\},
\een
for $\epsilon$ small. Moreover, since $p\ge 2$, one easily checks that, when $\beta$ is sufficiently small, any solution of $F(D,A)=1$ must satisfy $\|A-A^1\|_{\real^M}<\epsilon$.

If there existed $U^0\in G^+$, by theorem \ref{atractivoM}, $U^0$ would be of the form
\ben
U^0=(a_ie^{i\theta_i}Q(\cdot+y))_{1\le i\le M}, \ A^0=(a_i)_{1\le i\le M}\in S^+
\een
and $A^0$ would be a solution of \eqref{mina_0b_0}.  By uniqueness, $A^0=A(D)$, if $\|D\|_{\mathcal{S}_0}=\max_{i\neq j}\{k_{ij}\}<\delta$. Therefore
\ben
I(\mathcal{Q})<\sum_{i=1}^M (A(D))_i^2 I_1(Q) = I(U^0)
\een
which contradicts $U^0\in G$. Therefore $G^+$ is empty.

If there exists $U^0\in G$ with at least two nonzero components, the vector of nonzero components of  $U^0$, $(U^0)^+$, has to be a fully nontrivial ground-state for a  (L-NLS) system, with $2\le L\le M$. Applying the above argument, we obtain a contradiction. Therefore any ground-state has exactly one nontrivial component, which must be a scalar multiple of $Q$. A simple comparison of the action of such solutions proves the characterization \eqref{caracG}.
\end{proof}
\end{section}
\begin{section}{Applications to the (M-NLS) system}
We recall that we are always assuming (P1). Define
\ben
D=\{U\in (H^1(\real^N)^M: J(U)>0\}
\een
and, for each $U\in (H^1(\real^N))^M$,
\ben
M(U)=\sum_{i=1}^M \|u_i\|_2^2, \quad T(U)=\sum_{i=1}^M \|\nabla u_i\|_2^2, \quad 
E(U)=\frac{1}{2}T(U) -\frac{1}{2p+2}J(U)
\een
\ben
GN(U)=\frac{M(U)^{p+1-\frac{Np}{2}}T(U)^{\frac{Np}{2}}}{J(U)}.
\een

\begin{prop}
The set of solutions for the minimization problem 
\ben\label{mingagliardo}
GN(U)=\min_{W\in D} GN(W), \quad U\in D
\een
is $G$, up to scalar multiplication and scaling.
\end{prop}

\begin{proof}
By lemma \ref{caracterizacao} and theorem \ref{existencia}, we know that $G\neq\emptyset$ is the set of solutions of
\ben
I(U)=\min_{J(W)=\lambda_G} I(W),\quad J(U)=\lambda_G.
\een
Let $\mathcal{Q}\in G$ and $W\in D$. Since $I(\mathcal{Q})=J(\mathcal{Q})>0$, we have $\mathcal{Q}\in D$. Define
\ben
\nu = \left(\frac{J(\mathcal{Q})M(W)}{M(\mathcal{Q})J(W)}\right)^{\frac{1}{2p}}
\een
and
\ben
\zeta = \left(\nu^2\left(\frac{M(W)}{M(\mathcal{Q})}\right)\right)^{\frac{1}{N}}.
\een
Then $Z(x)=\nu W(\zeta x)$ satisfies
\ben
J(Z)=J(\mathcal{Q}), \quad M(Z)=M(\mathcal{Q}) \quad G(Z)=G(W).
\een
By the minimality of $\mathcal{Q}$, $I(\mathcal{Q})\le I(Z)$, which implies that $GN(\mathcal{Q})\le GN(Z)=GN(W)$. Therefore $\mathcal{Q}$ is a solution of  \eqref{mingagliardo}. On the other hand, if $W$ is a solution of \eqref{mingagliardo}, then one has necessarily $GN(Z)=GN(\mathcal{Q})$, which implies that $I(Z)=I(\mathcal{Q})$. Therefore $Z\in G$, which concludes our proof.
\end{proof}

Set
\ben
C_M = GN(\mathcal{Q})^{-1}, \ \mathcal{Q}\in G.
\een

\begin{cor}
The optimal constant for the vector-valued Gagliardo-Nirenberg inequality
\ben\label{GN}
\sum_{i,j=1}^M k_{ij}\|u_iu_j\|_{p+1}^{p+1} \le C\left(\sum_{i=1}^M \|u_i\|_2^2\right)^{p+1-\frac{Np}{2}}\left(\sum_{i=1}^M \|\nabla u_i\|_2^2\right)^{\frac{Np}{2}},\ U=(u_1,...,u_M)\in (H^1(\real^N))^M
\een
is $C_M$.
\end{cor}

\begin{nota}
Using proposition \ref{propnguyen}, we can determine, in particular, the constant $C_M$ presented by Nguyen et al. (\cite{nguyen}).
\end{nota}

We now focus on the \textbf{critical case} $p=2/N$. 
\begin{nota}
Let $\mathcal{Q}\in A$. The Pohozaev identity
\ben
\frac{N-2}{2}T(\mathcal{Q}) + \frac{N}{2}M(\mathcal{Q}) = \frac{N}{2p+2}J(\mathcal{Q}),
\een
together with $T(\mathcal{Q})+M(\mathcal{Q})=I(\mathcal{Q})=J(\mathcal{Q})$ and $p=2/N$, implies that
\ben
E(\mathcal{Q})=0.
\een
Therefore
\ben
GN(\mathcal{Q})=\frac{M(\mathcal{Q})^{\frac{2}{N}}}{p+1}, \forall \mathcal{Q}\in G.
\een
\end{nota}

From the vector-valued Gagliardo-Nirenberg inequality, we have the following optimal global existence result for (M-NLS):
\begin{prop}\label{global}
Suppose that $V_0\in (H^1(\real^N))^M$ is such that
\ben
M(V_0) < \left(\frac{p+1}{C_M}\right)^{\frac{N}{2}} = M(\mathcal{Q}),
\een
with $\mathcal{Q}\in G$. Then $T_{max}(V_0)=\infty$.
\end{prop}
\begin{proof}
It is a well-known fact that the functionals $M$ and $E$ are preserved by the flow generated by (M-NLS). Hence, if $V$ is the solution of (M-NLS) with initial data $V_0$, we have, by \eqref{GN},
\ben
E(V_0)=E(V(t))=\frac{1}{2}T(V(t)) -\frac{1}{2p+2}J(V(t))\ge \left(\frac{1}{2}-\frac{1}{2p+2}C_MM(V_0)^{\frac{2}{N}}\right)T(V(t)).
\een
Therefore $T(V(t))$ is bounded and so $T_{max}(V_0)=\infty$.
\end{proof}

\begin{nota}\label{gradmenorenergia}
It is easy to see that, for any $U\in (H^1(\real^N))^M$,
\ben
\left(1-\frac{1}{p+1}C_M\left(\sum_{i=1}^M \|u_i\|_2^2\right)^{\frac{2}{N}}\right)\left(\sum_{i=1}^M\|\nabla u_i\|_2^2\right)\le 2E(U).
\een
This inequality will be used later.
\end{nota}

The following result is an adaptation of the result in \cite{hmidikeraani} to the vector case.
\begin{lema}
Suppose that $\{U_n\}_{n\in\mathbb{N}}\subset (H^1(\real^N))^M$ verifies
\begin{enumerate}
\item $M(U_n)= C\ \forall n\in \nat$, for some $C>0$;
\item $T(U_n) = C'\ \forall n\in \nat$, for some $C' > 0$;
\item $E(U_n)\to 0$.
\end{enumerate}
Then, given $\delta_0>0$, there exists a subsequence $\{U_{n_k}\}$, $y_k\in\real^N$ and $R>0$ such that
\ben
\sum_{i=1}^M\int_{y_k+B_R} |(U_{n_k})_i|^2 dx \ge M(\mathcal{Q}) - \delta_0,
\een
where $\mathcal{Q}\in G$. If $C=M(\mathcal{Q})$ and $C'=T(\mathcal{Q})$, then $U_n\to U$ in $(H^1(\real^N))^M$ and $U\in G$.
\end{lema}

\noindent\textit{Sketch of the proof:}
The main ideas are the same as in the proof of theorem \ref{existencia}. Fixed $\epsilon>0$ and $1\le i\le M$, one splits each sequence $\{(U_n)_i\}$ into a set of bubbles $\{(U_n^l)_i\}_{1\le l\le L_i}$ using the concentration-compactness principle. Setting $L:=\max L_i$, define, for each $i$, $(U_n)_i^l=0$ if $L_i<l\le L$. Afterwards, one groups the bubbles into $L$ clusters in the same way as before and define $U_n^l$ as the vector of bubbles from cluster $L$. From the concentration-compactness principle and from the way we grouped the bubbles, we have
\ben
\left|M(U_n) - \sum_{l=1}^L M(U_n^l)\right|<\delta(\epsilon),\ T(U_n) \ge \sum_{l=1}^L T(U_n^l)-\delta(\epsilon),
\een
and

\ben
\left| J(U_n) - \sum_{l=1}^L J(U_n^l)\right|\le \delta(\epsilon).
\een

Now suppose that $M(U_n^l)<\left(\frac{p+1}{C_M}\right)^{\frac{N}{2}}-\delta_0$, for any $l$. Then, by the vector-valued Gagliardo-Nirenberg inequality and by remark \ref{gradmenorenergia},
\begin{align*}
J(U_n) &\le \sum_{l=1}^L J(U_n^l) + \delta(\epsilon) \lesssim \sum_{l=1}^L T(U_n^l) + \delta(\epsilon) \\&\lesssim \sum_{l=1}^L \left(1-\frac{1}{p+1}C_MM(U_n^l)^{\frac{2}{N}}\right)^{-1}E(U_n^l) + \delta(\epsilon) \lesssim E(U^n) + \delta(\epsilon).
\end{align*}
However, $E(U_n)\to 0$ and $J(U_n)\to (p+1)C'$, which is absurd. This proves the first part of the result.

If $C=\left(\frac{p+1}{C_M}\right)^{\frac{N}{2}}$, then the above argument shows that there can only exist one cluster and that all the components of the sequence $U_n$ verify the compactness alternative from the concentration-compactness principle. Since $\{U_n\}$ is bounded in $(H^1(\real^N))^M$, there exists $U\in (H^1(\real^N))^M$ such that $U_n\rightharpoonup U$ and, from the compactness alternative, it follows that $U_n\to U$ in $(L^2(\real^N)\cap L^{2p+2}(\real^N))^M$. In particular $M(U)=M(\mathcal{Q})$, $T(U)\le T(\mathcal{Q})$ and 
\ben
J(U)=\lim J(U_n) = (p+1)T(\mathcal{Q})=J(\mathcal{Q}).
\een
By the minimality of  $\mathcal{Q}$, we conclude that $U\in G$. Moreover, $T(U)=T(\mathcal{Q})=\lim T(U_n)$, which implies that $U_n\to U$ in $(H^1(\real^N))^M$. $\qedsymbol$

Using the previous lemma, one may prove the following results in the same way as in the scalar case $M=1$:
\begin{prop}[$L^2$ concentration]
Let $V_0\in (H^1(\real^N))^M$ be such that $T_{max}(V_0)<\infty$. Then, if $V$ is the corresponding solution of (M-NLS), there exists $x:[0,T(V_0))\to\real^N$ such that, for any $R>0$,
\ben
\liminf_{t\to T_{max}(V_0)} \sum_{i=1}^M \int_{|x-x(t)|<R} |(V(t))_i|^2\ge M(\mathcal{Q}), \ \mathcal{Q}\in G.
\een
\end{prop}
\begin{prop}[Blowup profile]
Let $V_0\in (H^1(\real^N))^M$ be such that $T_{max}(V_0)<\infty$ and $M(V_0)=M(\mathcal{Q})$, where $\mathcal{Q}\in G$. Let $V$ be the corresponding solution of (M-NLS). Then, for any sequence $t_n\to T_{max}(V_0)$, there exists $\mathcal{Q}_0\in G$ and $y_n\in\real^N$ such that
\ben
\left(\frac{T(\mathcal{Q}_0)}{T(V(t_n))}\right)^{\frac{N}{4}}V\left(\left(\frac{T(\mathcal{Q}_0)}{T(V(t_n))}\right)^{\frac{1}{2}}\cdot + y_n, t_n\right)\to \mathcal{Q}_0\mbox{ in } (H^1(\real^N))^M
\een
\end{prop}
\begin{nota}
We call the reader's attention to the fact that, throughout this section, we have only assumed (P1). If (P1) is false, then the left hand side of the vector-valued Gagliardo-Nirenberg inequality is not positive (which gives $C_M=0$), and, using proposition \ref{global}, one sees that all solutions of (M-NLS) are global. This implies that, in some sense, our results regarding the vector-valued Gagliardo-Nirenberg inequality, the $L^2$-concentration phenomena and the blowup profile are optimal.
\end{nota}

\end{section}

\begin{section}{Acknowledgements}
This work was partially supported by FCT (Portuguese Foundation for Science and Technology) through the grant SFRH/BD/96399/2013. We would like to thank Mário Figueira for having called our attention to this problem and for his helpful suggestions, encouragement and precious remarks.

\end{section}

\end{document}